\documentclass[12pt,a4paper,twoside]{article}
\usepackage[english]{babel}
\usepackage[latin1]{inputenc}
\usepackage{amssymb,amsfonts,amsmath,amsthm}
\usepackage{graphicx}
\usepackage[all]{xy}
\newcommand{\SortNoop}[1]{}

\newcommand{\pt}{\forall}

\newcommand{\sii}{\Leftrightarrow}

\newcommand{\plonge}{\hookrightarrow}

\newcommand{\mc}[1]{\mathcal{#1}}

\newcommand{\sd}{\sigma\delta}

\newcommand{\benum}{\begin{enumerate}}
\newcommand{\eenum}{\end{enumerate}}

\newtheorem{thm}{Theorem}[section]
\newtheorem{prop}[thm]{Proposition}
\newtheorem{cor}[thm]{Corollary}

\newtheorem{lm}[thm]{Lemma}

\begin{document}
\begin{titlepage}
\title{On $\sigma\delta$-Picard-Vessiot extensions}
\author{Ana Peón Nieto\footnote{Université Paris 7 Denis-Diderot, Équipe de Logique Mathématique, UFR de mathématiques, case 7012, site Chevaleret
75205 Paris Cedex 13 France, apeon@logique.jussieu.fr}}
\maketitle{}
\abstract{We study the differential Galois theory of difference equations under weaker hypothesis on the field of $\sigma$-constants. This framework yields a new approach to results by C.Hardouin and M.Singer, which answers possitively a question by M.Singer: under the classical hypothesis, the known results are still valid. In particular, our Galois group is isomorphic to theirs over a suitable field.  We also explicitly calculate the number of   connected components of the Galois group.}

\end{titlepage}
\tableofcontents
\newpage

\section{Introduction}
In \cite{HS} Hardouin and Singer study the Galois theory of linear difference-differential equations defined over a characteristic $0$ difference-differential field $k$ on which a set of commuting operators $\Sigma,\ \Delta$ and $\Pi$ act, and whose field of $\Sigma\Delta$-constants, $k^{\Sigma\Delta}$, is $\Pi$-differentially closed\footnote{Here $\Sigma$ and $\Pi$ are the sets of automorphisms and linear derivations respectively appearing in the equations under study, and $\Pi$ is a set of arbitrary derivations, the requirements on which are just commutativity with $\Delta\cup\Sigma$.}. In what follows, we answer a question by Michael Singer on whether the approach done in \cite{CHS} for the $\Sigma$-systems works also in this case, i.e., can we weaken the hypothesis on the field of constants without essentially altering the results? The answer is yes for the particular case $\Sigma=\{\sigma\},\ \Delta=\emptyset,\ \Pi=\{\delta\}$: supposing the field of $\sigma$-constants $C_k$ to be relatively algebraically closed inside $k$, we build a Galois group $\mathbb{H}$ that turns out to be isomorphic to the one in \cite{HS} once the constants have been suitably increased. Concerning the structure of the Galois group, we obtain an explicit description of the Picard-Vessiot extension $S$ as a direct sum of domains, which allows to calculate the number of connected components of the Galois group. Similarly, our group $\mathbb{H}$ is dense (in a certain sense) inside the Galois groups of difference equations naturally related to our difference-differential system. The methods used here follow those developped in \cite{CHS} for the difference case.
\\For the sake of comparison, we include a summary of the steps of \cite{HS} that can be contrasted with results in the present article:
\benum
\item Hardouin-Singer:
\benum
\item Uniqueness of the Picard-Vessiot extension for the given system of $\Sigma\Delta$-equations, $S=k\{Z,det(Z)^{-1}\}$.
\item Galois correspondence.
\item Association of a $\Sigma\Delta$-Picard-Vessiot extension $S_0=k[Z,det(Z)^{-1}]$, satisfying that $Gal_{\Sigma\Delta\Pi}(S/k)\subset Gal_{\Sigma\Delta}(S_0/k)$ is a dense subgroup in the Kolchin topology (c.f. Proposition 6.21 in \cite{HS} and Propostion \ref{intermediategroups} in here).
\item $\Sigma$-simplicity of $S$(c.f. Corollary 6.22 in \cite{HS}).
\eenum
\item Present:
\benum
\item Explicit construction of a $\sd$-ring $R$ which is a Picard-Vessiot extension over $kC_L$ for $\sigma(X)=AX$, and which is simple as a $\sigma$- ring, being the union of simple $\sigma$-rings $R_i$ (see Proposition \ref{espicardvessiot}).
\item Description of the differential group $\mathbb{H}(C_L)=Aut_{\sd}(R/kC_L)$ in terms of the algebraic groups $\mathbb{H}_i(C_L)=Aut_\sigma(R_i/kC_L)$ (see Proposition 6.21 in \cite{HS} and \ref{intermediategroups} here)
\item If the field $K=k(C)$ is generated over $k$ by $C$, a differentially closed field containing $C_L$ and on which $\sigma$ is the identity, then $S=R\otimes_{kC_L}K$ is a simple $\sigma$-ring and is therefore isomorphic to the (unique) Picard-Vessiot extension of Hardouin-Singer (see 6.16 in \cite{HS}, and Theorem \ref{decdirectsum} here). We then have $Aut_{\sd}(S/K)=\mathbb{H}(C)$.
\item A finer study of $R$ and of its field of fractions $L$ gives the number of idempotents of $S$ and the index of the connected component $\mathbb{H}^0$ of $\mathbb{H}$ (c.f. Theorem \ref{decdirectsum}, Lemma \ref{haciaconcom} and Proposition \ref{concom} in here).
\eenum
\eenum

I wish to thank my advisor, Zoé Chatzidakis, for invaluable discussions and support during all this time.

\section{The main result}
\underline{Notation and terminology}\\
All fields are supposed to be of characteristic $0$. Let $k$ be a field/ring with an automorphism $\sigma$ and a derivative $\delta$.
\begin{displaymath}
\begin{array}{ll}\nonumber 
qr(A), A^*&\textrm{the total ring of quotients of the ring }A\\\nonumber
frac(D)&\textrm{the fraction field of a domain }D\\\nonumber
k^a& \textrm{the algebraic closure of the field }k\\\nonumber
C_k/Fix(\sigma)(k)&\{ x\in k:\ \sigma(x)=x\}\textrm{ for a given field/ring }k\\\nonumber
D_k& C_k^a\cap k\textrm{, the relative algebraic closure of } C_k\textrm{ in }k.\\\nonumber
Aut_{\sigma}(K/k)&\textrm{ the group of $k$-$\sigma$-automorphisms of the ring $K$}\\\nonumber
Aut_{\sd}(K/k)& \textrm{the group of $k$-$\sd$-automorphisms of the ring $K$}\\\nonumber
k\{a_1,\dots,\ a_n\} &\textrm{the $\delta$-differential ring generated over $k$ by the $n$-tuple ($a_1,\dots,a_n$)}\\\nonumber
k\langle a_1,\dots,\ a_n\rangle &\textrm{the $\delta$-differential field generated over $k$ by the $n$-tuple ($a_1,\dots,a_n$)}\\\nonumber
\end{array}
\end{displaymath}

\begin{prop}\label{espicardvessiot}(Construction of the Picard-Vessiot extension.)
Let $k$ be a $\sigma\delta$-field of characteristic 0 on which $\sd=\delta\sigma$. Let $A\in GL_m(k)$, and consider the field $L=k\langle X, det (X)^{-1}\rangle$ (for some $m\times m$ matrix of indeterminates $X$) on which we extend the action of the automorphism by setting: 
\begin{eqnarray}\label{ecuacion}
\sigma(X)&=&AX
\end{eqnarray}
and commutativity of $\sigma$ and $\delta$. Suppose that $k$ satisfies: \[D_k=C_k\] (i.e., the constant field of $k$ is relatively algebraically closed in k.) Then $kC_L\{X,det(X)^{-1}\}$ is a Picard-Vessiot extension over $kC_L$ for the equation above, and is a simple $\sigma$-ring.
\end{prop}
\begin{proof}

For the proof we will consider the intermediate $\sigma$-rings and fields:
\[
K_n=k\big(X,det(X)^{-1},X',\dots,X^{(n)}\big)
\]
\[
F_n=kC_{K_n}[X,det(X)^{-1},X',\dots,X^{(n)}]
\]
 on which the action of $\sigma$ is defined by restriction. We have thus:
\begin{equation}\label{system}
\sigma(X^{(j)})=\sum_{i=0}^{j}\binom{j}{i}A^{(i)}X^{(j-i)}\qquad \pt j=0,1,\dots, n
\end{equation}
We will prove that the $F_n$'s are simple $\sigma$-rings $\pt n\in\mathbb{N}$. Once this has been proved we are done. Indeed, if $I\subset kC_L\{X,det(X)^{-1}\}$ is a non-zero proper $\sigma$-ideal, for all $n\in\mathbb{N}\quad I\cap kC_{K_n}[X,det(X)^{-1},X',\dots,X^{(n)}]$ is a $\sigma$-ideal. Since it is non-zero, 
\[
I\cap kC_{K_n}[X,det(X)^{-1},X',\dots,X^{(n)}]\neq 0
\]
for some $n\in\mathbb{N}$. But this implies that $I\cap F_n\neq 0$ for some $n\in\mathbb{N}$. Thus $1\in I\cap F_n\subset I$. So $kC_L\{X,det(X)^{-1}\}$ is a simple $\sigma$-ring, thus a simple $\sd$-ring, generated over $kC_L$ (as a $\delta$-ring) by a fundamental solution of the $\sigma$-equation, and so a Picard-Vessiot extension over $kC_L$.

To check the simplicity of $F_n$, we will consider the $kC_{K_n}$'s as $\sigma$-fields. Then (\ref{system}) defines a system of difference equations over $kC_{K_n}$, with the following associated matrix:
\begin{displaymath}
\mathcal{A} =
\left( \begin{array}{cccccccc}
A & 0 & 0& 0& \ldots&\quad&0 \\
A'&  A & 0& 0& \ldots&\quad&0\\
A''& 2A' & A&0& \ldots &\quad&0\\
A'''&3A''&3A'& A& \ldots&\quad& 0\\
\vdots&\vdots&\vdots&\vdots&\ddots&\quad&\vdots\\
A^{(n)}&\binom{n}{1}A^{(n-1)}&\binom{n}{2}A^{(n-2)}&\binom{n}{3}A^{(n-3)}&\ldots&\quad&A
\end{array} \right)
\end{displaymath}
Note that $(X,\dots,X^{(n)})^T$ is an obvious solution of (\ref{system}).
\\
\underline{Claim:} there is a fundamental matrix for the system ($\ref{system}$) with entries in $F_n$.
Indeed, consider the following matrix, defined by blocks:
\begin{displaymath}
\mathcal{X} =
\left( \begin{array}{ccccccc}
X & 0& 0& 0&\ldots&  \quad &0\\
X' &X& 0 &0   &\ldots&  \quad &0 \\
X''& 2X'& X & 0& \ldots&  \quad &0\\
X''' &3X'' &3X' &X &  \ldots&\quad& 0\\
\vdots& \vdots& \vdots& \vdots& \ddots& \quad&\vdots\\
X^{(n)} &\binom{n}{1}X^{(n-1)}& \binom{n}{2}X^{(n-2)}& \binom{n}{3}X^{(n-3)}& \ldots&\quad&  X
\end{array} \right)
\end{displaymath}
It is easily checked that $\mathcal{X}$ is a solution.
\\Given that $C_k=D_k$, Corollary 4.12 in \cite{CHS} 
 implies that $F_n/kC_{K_n}$ is a Picard-Vessiot extension.
\end{proof}

We can now proceed to the description of the group $Aut_{\sd}(R^*/kC_L),\ R^*=qr(R)$. The strategy followed is inspired by the one found in \cite{CHS}: we define the automorphism group independently of any considerations on the constants, as being the stabiliser of the differential locus of a generic solution; then, uniqueness of the Picard-Vessiot extension when the constants are differentially closed will imply that the group obtained is isomorphic to the one defined in \cite{HS} over a suitable field.
\begin{prop}Let 
\[
G=Aut_{\sd}(kC_L\langle X \rangle /kC_L)
\]
Consider the $\delta$-ideal $I_\delta\subset kC_L\{Y\}$ of differential polynomials over $kC_L$ that vanish at $X$. Let $GL_m(C_L)$ act on $kC_L\{Y\}$ by $\delta$-$kC_L$-automorphisms in the following way: to each $B\in GL_m(C_L)$ we associate the automorphim $g_B:\ Y\mapsto YB$.

Let $H=\{B\in GL_m(C_L):\ g_B\textrm{ leaves }I_\delta \textrm{ invariant }\}$.
Then $G\cong H$, so $G$ is the set of $C_L$-points of the linear differential group $\mathbb{H}$ defined over $kC_L$ by the property: ``$g_B$ leaves $I_\delta$ invariant''.

\end{prop}
\begin{proof}
It is easy to see that $G\plonge H$: given $g\in G$, its action is uniquely determined by the matrix $X^{-1}g(X)\in GL_m(C_L)$. Indeed, $g(X)$ is a fundamental matrix for the equation $\sigma(X)=AX$, hence $g(X)=XB_g$ for a unique $\ B_g\in GL_m(C_L)$. By commutativity with $\delta$, $g$ induces the desired action on $kC_L\{Y\}$. Finally, $G$ being a $\delta$-automorphism group over $kC_L$, it must leave $I_\delta$ invariant. So we have the desired inclusion.

Conversely, let $B\in H$. We may associate to this element the automorphism given by $g_B(X^{(n)})=\sum_{i=0}^n\binom{n}{i}X^{(i)}B^{(n-i)}$. Both maps are clearly the inverse of one another, and so they define an isomorphism and its inverse. It is clearly a differential morphism, and so induces a differential structure on $G$.
\end{proof}
\begin{cor}\label{intermediategroups} Let $\mathbb{H}=Aut_{\sd}(kC_L\langle X\rangle/kC_L)$ be the Galois group described above for the system of difference equations (\ref{ecuacion}). Then $\mathbb{H}$ is defined by
\begin{displaymath}
B\in\mathbb{H}\sii\pt n\in \mathbb{N}\left(\begin{array}{cccc}
                                                     B&0&\ldots&0\\
                                                     B'&B&\ldots&0\\
                                                     \vdots&\vdots&\ddots&\vdots\\
                                                     B^{(n)}&\cdots & &B
\end{array}\right)\in\mathbb{H}_n
\end{displaymath}
\begin{displaymath}
\sii\exists n_0:\pt n\in \{0,\dots, n_0\}\left(\begin{array}{cccc}
                                                     B&0&\ldots&0\\
                                                     B'&B&\ldots&0\\
                                                     \vdots&\vdots&\ddots&\vdots\\
                                                     B^{(n)}&\cdots & &B
\end{array}\right)\in\mathbb{H}_n
\end{displaymath}
where $\mathbb{H}_n$ is the Galois group $Aut_\sigma(F_nC_L/kC_L)$ for the difference system (\ref{system}). Furthermore, the embedding
\begin{eqnarray}\nonumber
i_n:\mathbb{H}&\plonge&\mathbb{H}_n\\\nonumber
B&\mapsto& (B,\dots,B^{(n)})
\end{eqnarray}
has Zariski dense image.
\end{cor}
\begin{proof}(c.f. Theorem 2.9 and Proposition 4.15 on \cite{CHS})
View now $kC_L\langle X\rangle$ and $kC_L\{Y\}$ as rings. The action of $B\in GL_m(C_L)$ on $X^{(n)}$ and $Y^{(n)}$ is now defined by:
$$
X^{(n)}\mapsto \sum_{j=0}^n\binom{n}{j} X^{(n-j)}B^{(j)}\qquad Y^{(n)}\mapsto \sum_{j=0}^n\binom{n}{j} Y^{(n-j)}B^{(j)}
$$
that is, via the action of $(B,\dots,B^{(n)})$. It follows then that $B$ leaves the differential ideal $I_\delta$ invariant if and only if $\pt n\in\mathbb{N}$ $(B,\dots,B^{(n)})$ leaves the ideal $I_n $ invariant (where $I_n\subset kC_L[Y_0,\dots,Y_n]$ is the ideal of polynomials vanishing at $(X,\dots,X^{(n)})^T$) if and only if for each $n$ $(B,\dots,B^{(n)})$ leaves the ideal of $kC_L$-polynomials vanishing at a fundamental solution of (\ref{system}) invariant, if and only if for each $n$ 
\begin{displaymath}
\left(\begin{array}{cccc}
                                                     B&0&\ldots&0\\
                                                     B'&B&\ldots&0\\
                                                     \vdots&\vdots&\ddots&\vdots\\
                                                     B^{(n)}&\cdots & &B
\end{array}\right)\in\mathbb{H}_n
\end{displaymath}
This proves the first equivalence, the second being an easy consequence of noetherianity of the Kolchin topology.
\\The last assertion follows from the definitions of $\mathbb{H}$ and $\mathbb{H}_n$ (as being respectively the stabiliser of the differential and algebraic loci of a fundamental solution to the given equation), and the preceeding arguments.
\end{proof}

\begin{lm}\label{esfinito}
Let $L,D_L,C_L$ be as defined in \ref{espicardvessiot}. Then $(kC_L)^{a}\cap L/kC_L$ is a finite extension; in particular $[D_L:C_L]=l<\infty$
\end{lm}
\begin{proof}
Consider $I=\mc{I}_{\delta}(X/kC_L)$, the ideal of differential polynomials over $kC_L$ that vanish at $X$. It is a prime ideal of $kC_L\{Y\}$, and so the differential set it defines (say $W=\mc{V}_\delta(I)$) is a Kolchin closed set; since the Kolchin topology is noetherian, $W$ has a finite number of absolutely irreducible components; equivalently, there is only a finite number of minimal ideals of $(kC_L)^a\{Y\}$ containing $I$. Since there is a one-to-one correspondence between such minimal ideals and $kC_L$-$\delta$-embeddings of $(kC_L)^{a}\cap L$ into $(kC_L)^{a}$, the result follows.
\end{proof}
\begin{thm}\label{decdirectsum}
Let $\hat{C}$ denote the $\sd$-field whose underlying $\delta$-field is a differential closure of $C_L$ (taken to be linearly disjoint from $L$ over $D_L$) on which the automorphism $\sigma$ acts as the identity. Then:
\benum
\item $S=R\otimes_{kC_L} k\hat{C}$ is the unique Picard-Vessiot extension over the $\sd$-field $k\hat{C}$ for the system of equations $\sigma(X)=AX$.
\item $S\cong\oplus_{i=0}^{l-1} S_i$ where $l=[D_L:C_L]$, $\sigma(S_i)=S_{i+1}$, and each $S_i$ is simple as a $\sigma^l\delta$-domain. The $\sigma^l$-ring $S_0$ is equal to $R[k\hat{C}]\cong R\otimes_{kD_L}k\hat{C}$, where $\sigma^l$ is the identity on $\hat{C}$.
\eenum
 \end{thm}
\begin{proof}
$(1)$ By Lemma 1.11 in \cite{SVDP}
, it follows that $S$ is a simple $\sd$-ring. Indeed, $R\otimes_{C_L}\hat{C}$ is $\sigma$-simple by Lemma 1.11 in \cite{SVDP}
hence it is also $\sd$-simple. Since $k\otimes_{C_L}\hat{C}$ is a domain by the hypothesis on $\hat{C}$, we can localise without modifying the $\sigma$-simplicity. By Proposition 6.16 in \cite{HS}
, we then have that $S\cong (R\otimes_{C_L}\hat{C})_{k\otimes_{C_L}\hat{C}}$.
\\\\$(2)$ All assertions follow from Lemma 6.8 in \cite{HS} 
, except for $l=[D_L:C_L]$. Let $D_L=C_L[\alpha]$, and let $p(y)$ be the minimal polynomial of $\alpha$ over $C_L$. Since $D_L$ and $k$ are linearly disjoint over $D_k=C_k$, this polynomial remains irreducible over $kC_L$. Thus, if $D_L'\subset C_0$, where $C_0$ is the algebraic closure of $C_L$ inside $\hat{C}$ and $D_L'$ is the subfield of $C_0$ isomorphic to $D_L$:
\[
A=R\otimes_{kC_L}kD_L'\cong R[y]/\left(p(y)\right)
\]
By Corollary 4.12 in \cite{CHS}
, $R$ is $\sigma^l$-simple, and $D_L=Fix(\sigma^l)(R^*)=Fix(\sigma^l)(R)$ by Lemma 4.6 in \cite{CHS}
, so $p(y)$ splits completely over $R$, which implies that the ring  $A$ is isomorphic to the direct sum of $l$ copies of $R$, that is, $A$ has $l$ primitive idempotents, $e_0,\dots, e_{l-1}$. Now, inside $A$ there are two ``distinct copies'' of the field $D_L$, namely $D_L\otimes 1$ and $1\otimes D_L'$; as $\sigma^l$-fields, they are isomorphic. Moreover, for each $i$, $e_i(D_L\otimes 1)=e_i(1\otimes D_L')$. This follows from the fact that $e_iA$ is a domain, and $D_L/C_L$ is Galois.\\
So
\[
R\otimes_{kC_L}k\hat{C}\cong\oplus_{i=0}^{l-1}e_i A\otimes_{kD_L}k\hat{C}
\]
All there is left to do is to check that $A\otimes_{kD_L}k\hat{C}$ is a domain, for then Lemma 6.8 in \cite{HS} 
 implies that the $\sd$-structure  induced by the previous isomorphism is the right one. Now, since $kC_0=kD_L'[C_0]$ ($C_0/D_L'$ is algebraic), then $R\otimes_{kD_L}kC_0\cong R\otimes_{D_L}C_0$ is a domain. In the same fashion, since $k\hat{C}/kC_0$ is regular (by Lemma 6.11 in \cite{HS} 
  and (1)), then $\left(k\otimes_{D_L}C_0\right)\otimes_{kC_0}k\hat{C}$ is also a domain.
\end{proof}

\begin{cor}\label{corfinal}
Let $k,\ L,\ R,\ \hat{C}$ be as above, and suppose that $C_k$ is differentially closed. Let $P$ be a Picard-Vessiot extension over $k$ for the system of equations in \ref{espicardvessiot}, and let $\mc{G}$ denote the differential group defined in Proposition 6.18 in \cite{HS} 
 (so that we have that $Aut_{\sd}(P^*/k)=\mc{G}(C_k)$). Then $\mc{G}$ and $\mathbb{H}$ are isomorphic over $\hat{C}$.
\end{cor}
\begin{proof}
Consider the tensor product:
\[
P\otimes_{C_k}\hat{C}
\]
By Corollary 6.22 in \cite{HS} 
 $P$ is a simple $\sigma$-ring. So  Lemma 1.11 in \cite{SVDP} 
  implies that $P\otimes_{C_k}\hat{C}$ is simple as a $\sigma$-ring, and so also as a $\sd$-ring. Furthermore, since $k$ and $\hat{C}$ are linearly disjoint over $C_k$, $P$ is finitely generated over the domain $k\otimes_{C_k}\hat{C}$, so one may localise to obtain a Picard-Vessiot extension $(P\otimes_{C_k}\hat{C})_{k\otimes \hat{C}}/k\hat{C}$ for the system of equations in \ref{espicardvessiot}. By uniqueness of the Picard-Vessiot extension when the $\sigma$-constants are differentially closed, together with theorem \ref{decdirectsum}, one must have
\[
\tilde{P}=(P\otimes_{C_k}\hat{C})_{k\otimes \hat{C}}\cong(P\otimes_{k\hat{C}}k\hat{C})\cong(R\otimes_{C_L}\hat{C})_{kC_L\otimes \hat{C}}=\tilde{R}
\] 
This isomorphism induces an isomorphism $Aut_{\sd}(\tilde{P}^*/k\hat{C})\cong Aut_{\sd}(\tilde{R}^*/k\hat{C})$ via conjugation. So all there's left to check is that 
\[Aut_{\sd}(\tilde{P}^*/k\hat{C})=\mc{G}(\hat{C})\]\[ Aut_{\sd}(\tilde{R}^*/k\hat{C})=\mathbb{H}(\hat{C})\]
The first follows from the discussion in the paragraph following Proposition 6.18 in \cite{HS} 
 and the fact that:
\[
Aut_{\sd}(\tilde{P}^*/k\hat{C})=Aut_{\sd}(\tilde{P}/k\hat{C})=Aut_{\sd}(P\otimes_{C_k}\hat{C}/k\otimes_{C_k}\hat{C})
\]
For the second, we must show that:
\[
\{B\in GL_m:\ g_B \textrm{ leaves } I_\delta(X/k\hat{C}) \textrm{ invariant}\}=\mathbb{H}(\hat{C})
\]
For that, it suffices to prove that $I_\delta(X/k\hat{C})$ is still defined over $kC_L$. Now note that to show that $Y$ and $X$ define the same differential locus, one may prove that for every $n\in\mathbb{N}$ $(X,\dots,\ X^{(n)})$ and $(Y,\dots,\ Y^{(n)})$ have the same algebraic locus (i.e., the differential locus of $X$ over $k\hat{C}$ is totally defined by the algebraic loci $W_n=\mc{V}(X,\dots,X^{(n)}/k\hat{C})$). But $kC_L$ and $\hat{C}$ are linearly disjoint over $C_L$ by Lemma 6.11 in \cite{HS}
; hence, $W_n$ is still defined over $kC_L\ \pt n\in\mathbb{N}$.

\end{proof}
\section{A closer look at the Galois group}
We investigate in this section the structure of the Galois group defined in the previous sections. It turns out that the analysis done in \cite{SVDP} works also in this case:
\begin{lm}\label{haciaconcom}
Let $k$ be a $\sd$-field such that $C_k=D_k$. Let $R$ be the Picard-Vessiot extension built in \ref{espicardvessiot}, $L$ its fraction field and $C_L$ the constant field of $L$. Consider the $\sd$-field $C$, whose underlying set is $C_L^{a}$ and on which the automorphism acts as the identity, and the derivation extends in the only possible way from $\delta|_{C_L}$. Then:
\benum
\item[i)]$R\otimes_{C_L}C$ is a Picard-Vessiot extension over $kC$ whose total ring of fractions is $L\otimes_{C_L}C$.
\item[ii)] If we set $R\otimes_{C_L}C\cong\oplus_{i=0}^{l-1}R_i$, we have that $R_0$ is Picard-Vessiot over $kC$ for $\sigma^lX=A_lX$ where $A_l=\sigma^{l-1}(A)\sigma^{l-2}(A)\dots\sigma(A)A$; if $R_0^*=qr(R_0)$, then $R_0^*\cong L\otimes_{D_L}C$ as a $\sigma^l\delta$-ring.
\eenum
\end{lm}
\begin{proof}
For the first assertion of $i)$, since $R$ is $\sigma$-simple, we may apply Lemma 1.11 in \cite{SVDP} 
 to deduce that $R\otimes_{C_L}C$ is $\sigma$-simple, hence also $\sd$-simple; finally, the fact that $kC\subset R\otimes_{C_L}C$ yields the result.
\\The decomposition of $R\otimes_{C_L}C$ into a direct sum follows from \ref{decdirectsum}. By Lemma 6.8 in \cite{HS} 
 each of the summands is $\sigma^l\delta$-simple and (using the decomposition and the fact that $kC\subset R_i$) isomorphic to $kC\{X,det(X)^{-1}\}$, yielding $ii)$ and the second assertion of $i)$.
\end{proof}

\begin{prop}\label{concom}
Let $R,\ k, C_L$ be as above. Let $C$ be the difference differential field whose underlying differential field is $C_L^a$, and on which $\sigma$ acts as the identity. Consider 
\[
Aut_{\sd}(L\otimes_{C_L}kC/kC)=\mathbb{H}(C)
\]
There is an exact sequence:
\begin{displaymath}
0\to Aut_{\sigma^l\delta}(R_0^*/kC)\xrightarrow{\Gamma} \mathbb{H}(C)\xrightarrow{\Delta} \mathbb{Z}/l\mathbb{Z}\to 0
\end{displaymath}
where $R_0^*=frac(R_0)\cong L\otimes_{D_L}k(C)$, $l=[D_L:C_L]$ and $\Gamma$  can be chosen to be a difference-differential group morphism.
\end{prop}
\begin{proof}
The proof is the same as the one of  Corollary 1.17 in \cite{SVDP}, modulo small modifications concerning the extra differential structure involved, as well as the fact that $C$ is not necessarily differentially closed. The difficulties arise in relation with
\benum
\item definition of $\Gamma$,
\item commutativity of the latter and $\delta$,
\item surjectivity of $\Delta$,
\eenum 
but they can be easily solved by applying theorem \ref{decdirectsum}. 
\\The fact that $l=[D_L:C_L]$ follows form \ref{decdirectsum}, as well as lemma \ref{haciaconcom}.
\end{proof}

\newpage
\bibliographystyle{alpha}
\bibliography{finalversionnopreliminaries}

\end{document}